\documentclass[12pt]{amsart}
\usepackage{amssymb}
\usepackage{verbatim}
\usepackage{hyperref}
\usepackage{color}

\begin{document}

% define theorem environments
\newtheorem{theorem}{Theorem}    %[section]
\newtheorem{proposition}[theorem]{Proposition}
\newtheorem{conjecture}[theorem]{Conjecture}
\def\theconjecture{\unskip}
\newtheorem{corollary}[theorem]{Corollary}
\newtheorem{lemma}[theorem]{Lemma}
\newtheorem{sublemma}[theorem]{Sublemma}
\newtheorem{fact}[theorem]{Fact}
\newtheorem{observation}[theorem]{Observation}
\theoremstyle{definition}
\newtheorem{definition}{Definition}
\newtheorem{notation}[definition]{Notation}
\newtheorem{remark}[definition]{Remark}
\newtheorem{question}[definition]{Question}
\newtheorem{questions}[definition]{Questions}
\newtheorem{example}[definition]{Example}
\newtheorem{problem}[definition]{Problem}
\newtheorem{exercise}[definition]{Exercise}

\numberwithin{theorem}{section}
\numberwithin{definition}{section}
\numberwithin{equation}{section}

\def\reals{{\mathbb R}}
\def\torus{{\mathbb T}}
\def\heis{{\mathbb H}}
\def\integers{{\mathbb Z}}
\def\rationals{{\mathbb Q}}
\def\naturals{{\mathbb N}}
\def\complex{{\mathbb C}\/}
\def\distance{\operatorname{distance}\,}
\def\support{\operatorname{support}\,}
\def\dist{\operatorname{dist}\,}
\def\Span{\operatorname{span}\,}
\def\degree{\operatorname{degree}\,}
\def\kernel{\operatorname{kernel}\,}
\def\dim{\operatorname{dim}\,}
\def\codim{\operatorname{codim}}
\def\trace{\operatorname{trace\,}}
\def\Span{\operatorname{span}\,}
\def\dimension{\operatorname{dimension}\,}
\def\codimension{\operatorname{codimension}\,}
\def\nullspace{\scriptk}
\def\kernel{\operatorname{Ker}}
\def\ZZ{ {\mathbb Z} }
\def\p{\partial}
\def\rp{{ ^{-1} }}
\def\Re{\operatorname{Re\,} }
\def\Im{\operatorname{Im\,} }
\def\ov{\overline}
\def\eps{\varepsilon}
\def\lt{L^2}
\def\diver{\operatorname{div}}
\def\curl{\operatorname{curl}}
\def\etta{\eta}
\newcommand{\norm}[1]{ \|  #1 \|}
\def\expect{\mathbb E}
\def\bull{$\bullet$\ }
\def\det{\operatorname{det}}
\def\Det{\operatorname{Det}}
\def\multiR{\mathbf R}
\def\bestA{\mathbf A}
\def\Apq{\mathbf A_{p,q}}
\def\Apqr{\mathbf A_{p,q,r}}
\def\rank{\mathbf r}
\def\diameter{\operatorname{diameter}}
\def\bp{\mathbf p}
\def\bff{\mathbf f}
\def\bg{\mathbf g}
\def\essd{\operatorname{essential\ diameter}}

\def\mab{\max(|A|,|B|)}
\def\t2{\tfrac12}
\def\tatb{tA+(1-t)B}

\newcommand{\abr}[1]{ \langle  #1 \rangle}

\newcommand{\Norm}[1]{ \Big\|  #1 \Big\| }
\newcommand{\set}[1]{ \left\{ #1 \right\} }
\def\one{{\mathbf 1}}
\newcommand{\modulo}[2]{[#1]_{#2}}

\def\scriptf{{\mathcal F}}
\def\scripts{{\mathcal S}}
\def\scriptq{{\mathcal Q}}
\def\scriptg{{\mathcal G}}
\def\scriptm{{\mathcal M}}
\def\scriptb{{\mathcal B}}
\def\scriptc{{\mathcal C}}
\def\scriptt{{\mathcal T}}
\def\scripti{{\mathcal I}}
\def\scripte{{\mathcal E}}
\def\scriptv{{\mathcal V}}
\def\scriptw{{\mathcal W}}
\def\scriptu{{\mathcal U}}
\def\scripta{{\mathcal A}}
\def\scriptr{{\mathcal R}}
\def\scripto{{\mathcal O}}
\def\scripth{{\mathcal H}}
\def\scriptd{{\mathcal D}}
\def\scriptl{{\mathcal L}}
\def\scriptn{{\mathcal N}}
\def\scriptp{{\mathcal P}}
\def\scriptk{{\mathcal K}}
\def\scriptP{{\mathcal P}}
\def\scriptj{{\mathcal J}}
\def\scriptz{{\mathcal Z}}
\def\frakv{{\mathfrak V}}
\def\frakG{{\mathfrak G}}
\def\frakA{{\mathfrak A}}
\def\frakB{{\mathfrak B}}
\def\frakC{{\mathfrak C}}

\begin{comment}
\def\frakg{{\mathfrak g}}
\def\frakG{{\mathfrak G}}
\def\boldn{\mathbf N}
\end{comment}

\author{Michael Christ}
\address{
        Michael Christ\\
        Department of Mathematics\\
        University of California \\
        Berkeley, CA 94720-3840, USA}
\email{mchrist@math.berkeley.edu}
\thanks{Research supported in part by NSF grant DMS-0901569.}

%Any opinions, findings, and conclusions
%or recommendations expressed in this paper are those of the author
%and do not necessarily reflect the views of the National Science Foundation.}
%s DMS-0401260 and

\date{May 1, 2012. Revised July 20, 2012.}

\title {Near Equality In The Two-dimensional Brunn-Minkowski Inequality} 
\begin{abstract}
If a pair $(A,B)$ of subsets of $\reals^2$ nearly realizes equality in the
Brunn-Minkowski inequality, in the sense that the measure  of the associated
sumset $A+B$ is nearly equal to $(|A|^{1/2}+|B|^{1/2})^2$,
then $(A,B)$ must nearly coincide with a pair of homothetic convex sets.
\end{abstract}
\maketitle

\section{Introduction}

This is one of a series of papers 
\cite{youngdiscrete}, \cite{christradon}, \cite{christrieszsobolev}, \cite{christyoungest}
concerned with the interplay between linear structure, analysis, and affine-invariant geometry of Euclidean spaces.
Earlier papers in this series have built on a characterization of cases of near equality
in the one-dimensional Brunn-Minkowski inequality to deduce analytic consequences in both
one and higher dimensions. 
The purpose of the present work is to characterize cases of near equality in the Brunn-Minkowski
inequality itself, for subsets of $\reals^2$. 

Let $A,B\subset\reals^d$ be sets. 
Their algebraic sum $A+B$ is defined to be
\[A+B=\set{a+b: a\in A \text{ and } b\in B}.\] 
For $A\subset\reals^d$, and $t\ge 0$, 
\[tA+v=\set{ta: a\in A}.\]
Thus for $s,t\in[0,\infty)$, $sA+tB=\set{sa+tb: a\in A \text{ and } b\in B}$.
If $A,B$ are Borel measurable, then $A+B$ is Lebesgue measurable. 

The Brunn-Minkowski inequality states that for any two nonempty Borel sets $A,B\subset\reals^d$,
\begin{equation} \label{eq:BMadd} |A+B|^{1/d}\ge |A|^{1/d}+|B|^{1/d}.  \end{equation}
The assumption of Borel measurability ensures that the sumset $A+B$ is Lebesgue measurable;
there are other versions of the inequality for more general sets.
A valuable survey is the article of Gardner \cite{gardner}.
If $A,B$ are Borel sets with finite measures,
then equality holds in \eqref{eq:BMadd} if and only if there exist a compact convex set $\scriptk$,
scalars $\alpha,\beta\ge 0$, and vectors $u,v\in\reals^d$ such that
$A\subset \alpha\scriptk+u$, $B\subset\beta\scriptk+v$, 
$|(\alpha\scriptk+u)\setminus A|=0$ and $|(\beta \scriptk+v)\setminus B|=0$
\cite{henstock},\cite{HO}.

Near equality can arise only in the obvious way:
\begin{theorem} \label{thm:BMadd}
For any compact subset $\Lambda\subset(0,\infty)$ 
there exists a function $(0,\infty)\owns \delta\mapsto \eps(\delta)\in(0,\infty)$
satisfying $\lim_{\delta\to 0}\eps(\delta)=0$ with the following property.
Let $A,B\subset\reals^2$ be Borel sets with  positive Lebesgue measures
satisfying $|A|/|B|\in \Lambda$.  If
\begin{equation} |A+B|^{1/2}\le |A|^{1/2}+ |B|^{1/2}+\delta\mab^{1/2} \end{equation}
then there exist a compact convex set $\scriptk\subset\reals^2$,
scalars $\alpha,\beta\in\reals^+$, and vectors $u,v\in\reals^2$ such that
\begin{equation} A \subset \alpha \scriptk+u \text{ and } 
|(\alpha\scriptk+u)\setminus A|<\eps(\delta)\mab \end{equation}
while
\begin{equation} B \subset \beta\scriptk+v \text{ and } 
|(\beta\scriptk+v)\setminus B|<\eps(\delta)\mab. \end{equation}
\end{theorem}
Conversely, if $A,B$ are related to $\scriptk$ in the way indicated,
then $|A+B|^{1/2}$ is nearly equal to $|A|^{1/2}+|B|^{1/2}$.

One question not studied here is the nature of the function $\delta\mapsto\eps(\delta)$.
To the best of our understanding,
it is natural to conjecture that $\eps=O(\delta^\gamma)$ for some positive exponent $\gamma$. 
We hope to return to this point in the future.

The foundation of our analysis is the following sharp one-dimensional version of 
Theorem~\ref{thm:BMadd}.
\begin{proposition} \label{prop:nearlyintervals}
Let $A,B\subset\reals^1$ be Borel sets. If $|A+B|<|A|+|B|+\delta$ and if $\delta<\min(|A|,|B|)$
then there exist intervals $I,J$ such that $A\subset I$, $B\subset J$, $|I|<|A|+\delta$, and $|J|<|B|+\delta$.
\end{proposition}
The conclusion can of course be equivalently restated as upper bounds for the diameters of $A,B$.
Proposition~\ref{prop:nearlyintervals} is a continuum analogue of a theorem of Fre{\u\i}man \cite{freiman} 
concerning sums of finite subsets of $\integers$, and can be deduced as a corollary
of this discrete case. A more direct proof, adapted from a pre-existing proof of the discrete
version, is included in \cite{christrieszsobolev}. 

Our proof is organized as an induction on the dimension, 
in which $\reals^d$ is regarded as $\reals^k\times\reals^{d-k}$
and the induction hypothesis is applied in both factors $\reals^k,\reals^{d-k}$;
this is done for arbitrary rotations of $\reals^d$. Most of the work is devoted to
obtaining sufficient control over the sets  $A,B$ to allow application of a compactness
argument. In the limit $\delta\to 0$, one obtains a pair of sets which achieves
exact equality in the Brunn-Minkowski inequality. An application of the known characterization
of such equality then leads easily to our conclusion.
However, one step breaks down unless $k=1$, while another requires $d-k=1$, so this analysis is only completely
successful in dimension $d=2$. The higher-dimensional case is treated in the sequel \cite{christbmhigh},
using an additional ingredient, but the last few steps of the argument presented here
are also an essential part of the higher-dimensional analysis.

There has been significant recent work \cite{FMP2008}, \cite{FMP2010}
on stability in the isoperimetric inequality, which is a limiting form of the Brunn-Minkowski inequality. 

\begin{notation}
Throughout the analysis, we identify $\reals^d$ with $\reals^1\times\reals^{d-1}$, with coordinates $(x,y)$.
$|S|$ denotes the Lebesgue measure of a subset $S$ of $\reals^{d-1},\reals^d$, or $\reals^1$, according to context. 
If $T\subset\reals^d$, then for $x\in\reals^1$, $T_x$ denotes the set
\[T_x=\set{y\in\reals^{d-1}: (x,y)\in T}.\]
The projection $\pi:\reals^d\to\reals^1$ is the mapping $\reals^1\times\reals^{d-1}\owns (x,y)\mapsto x$.
For any set $A\subset\reals^d$ and any $x\in\reals^1$,
\[ A_x=\set{y\in\reals^{d-1}: (x,y)\in A}.  \]
The notation $O(\delta^\gamma)$, where $\gamma$ is some positive exponent,
indicates a quantity which is bounded above by $C\delta^\gamma$,
uniformly for all pairs $(A,B)$ of sets satisfying indicated hypotheses,
while $O(1)$ indicates a quantity which is bounded uniformly in $A,B,t,\delta$ 
provided that $\delta$ is sufficiently small and $t\in\Lambda$.
The notation $o_\delta(1)$ indicates a quantity which tends to zero as $\delta$ tends to zero,
again uniformly for all pairs $(A,B)$ of sets satisfying indicated hypotheses.
by $S\bigtriangleup T$ denotes the symmetric difference $(S\setminus T) \cup (T\setminus S)$ of two sets.
\end{notation}

\section{Two reductions} \label{section:tworeductions}
An equivalent formulation of the Brunn-Minkowski inequality is that
for all nonempty Borel sets $A,B$ and all scalars $t\in(0,1)$,
\begin{equation} \label{eq:BMmult} |tA+(1-t)B|\ge |A|^t|B|^{1-t}.  \end{equation}
Theorem~\ref{thm:BMadd} has a corresponding equivalent formulation, which will be more
convenient for our purpose.

\begin{theorem} \label{thm:BMmult}
For any compact subset $\Lambda\subset(0,1)$ 
there exists a function $(0,\infty)\owns \delta\mapsto \eps(\delta)\in(0,\infty)$
satisfying $\lim_{\delta\to 0}\eps(\delta)=0$ with the following property.
Let $A,B\subset\reals^2$ be Borel sets with finite, positive Lebesgue measures, and let $t\in\Lambda$. If
\begin{equation} |tA+(1-t)B|\le |A|^t |B|^{1-t}+\delta\mab \end{equation}
then there exist a compact convex set $\scriptk\subset\reals^2$
and vectors $a,b\in\reals^2$ such that
\begin{equation} A \subset \scriptk+a \text{ and } 
|(\scriptk+a)\setminus A|<\eps(\delta)\mab \end{equation}
while
\begin{equation} B \subset \scriptk+b \text{ and } 
|(\scriptk+b)\setminus B|<\eps(\delta)\mab. \end{equation}
\end{theorem}

Theorem~\ref{thm:BMadd} follows directly from Theorem~\ref{thm:BMmult}. Given arbitrary Borel sets $A,B$
with finite, positive Lebesgue measures, choose $t\in\reals^+$ so that $\scripta=t^{-1}A$ and $\scriptb=(1-t)^{-1}B$
satisfy $|\scripta|=|\scriptb|$.  If $(A,B)$ gives near equality in \eqref{eq:BMadd} then
$|t\scripta+(1-t)\scriptb|=|A+B|$ while
\[ (|A|^{1/d}+|B|^{1/d})^d =(t|\scripta|^{1/d}+(1-t)|\scriptb|^{1/d})^d 
=|\scripta|^t|\scriptb|^{1-t}.  \]
So long as the ratio $|A|/|B|$ is confined to a compact subset of $(0,\infty)$,
$\max(|A|,|B|)$ is uniformly comparable to $\max(|\scripta|,|\scriptb|)$. Therefore
near equality in $\eqref{eq:BMadd}$, for the pair $(A,B)$, in the form assumed in the statement of
Theorem~\ref{thm:BMadd}, implies the hypotheses of Theorem~\ref{thm:BMmult} for the pair $(\scripta,\scriptb)$.
Applying the Theorem~\ref{thm:BMmult} to $(\scripta,\scriptb)$,
yields the conclusion of Theorem~\ref{thm:BMadd} for $(A,B)$. Thus it suffices to prove Theorem~\ref{thm:BMmult}.

Implicit in the hypotheses of Theorem~\ref{thm:BMmult} is a condition on $|A|/|B|$. 
\begin{lemma} \label{lemma:goodratio}
For any compact set $\Lambda\subset(0,1)$ there exists 
a constant $C<\infty$ which depends only on the dimension $d$, with the following property
for all sufficiently small $\delta>0$.
Let $A,B$ be Borel subsets of $\reals^d$ with positive, finite Lebesgue measures.
If $t\in\Lambda$ and $|tA+(1-t)B|<|A|^t|B|^{1-t}+\delta\max(|A|,|B|)$ then 
\[ \Big|\frac{|A|}{|B|} -1\Big|\le C\delta^{1/2}. \]
\end{lemma}

\begin{proof}
By \eqref{eq:BMadd}  and the arithmetic-geometric mean inequality,
\begin{equation*} |tA+(1-t)B| \ge \big(|tA|^{1/d}+|(1-t)B|^{1/d}\big)^d = \big(t|A|^{1/d}+(1-t)|B|^{1/d}\big)^d 
\ge |A|^t|B|^{1-t}.\end{equation*} 
The arithmetic-geometric inequality holds in the sharper form
\[ta+(1-t)b\ge a^tb^{1-t} \Theta_t(a/b)\]
where $\Theta$ is a continuous function on $(0,\infty)$ which satisfies
$\Theta_t(r)\ge 1$ for all $r\in(0,\infty)$, $\Theta_t(r)=1$ only for $r=1$,
$\Theta_t(r)\to \infty$ as $r\to\infty$ and as $r\to 0$,
\[\Theta_t(r)\ge c(r-1)^2\] for all $r$ sufficiently close to $1$ for a certain positive constant $c$.
All of these bounds hold uniformly for all $t$ in any compact subset of $(0,\infty)$.
Thus \[|tA+(1-t)B|\ge \Theta_t(|A|/|B|) |A|^t|B|^{1-t},\]
so that if $|tA+(1-t)B|<|A|^t|B|^{1-t}+\delta\max(|A|,|B|)$ then $\Theta_t(|A|/|B|)\le 1+O(\delta)$. 
\end{proof}

It suffices to prove Theorem~\ref{thm:BMmult} under the assumption that $|A|=1$,
for the formulation is scale-invariant.
In the main body of the proof of Theorem~\ref{thm:BMmult}, we will assume
that $0<\delta\le 1$ is sufficiently small,
and that $A,B\subset\reals^2$ are nonempty compact sets which satisfy the hypotheses
\begin{equation} \label{hypos}
\left\{
\begin{aligned}
& \big|\,|A|-1\,\big|<\delta 
\ \text{ and } \ 
 \big|\,|B|-1\,\big|<\delta
\\ & |tA+(1-t)B| <1+\delta
\\ & t\in\Lambda
\end{aligned}
\right.
\end{equation}
where $\Lambda$ is a fixed compact subset of $(0,1)$. All conclusions hold uniformly for
all such $A,B,t$, for all $\Lambda$, but do depend quantitatively on $\Lambda$.
By virtue of Lemma~\ref{lemma:goodratio}, this will establish the theorem for arbitrary compact sets.
A supplementary argument at the end of the proof will extend the conclusion to all Borel sets.

At the outset we will consider general dimensions $d$, and will specialize to $d=2$ when 
the argument requires it.
Throughout the discussion, all bounds are implicitly asserted to be uniform in $t$ provided that 
$\Lambda$ remains fixed and $t\in\Lambda$. 

\section{Localization}

The purpose of this section is to establish a localization principle, to the effect that if a pair $(A,B)$
achieves near equality in \eqref{eq:BMmult}, then so do certain pairs of subsets.
Denote by $\sup S$ and $\inf S$ the supremum and infimum of any bounded set $S\subset\reals^1$.

\begin{lemma} Let $A,B\subset\reals^1$ be compact sets with finite Lebesgue measures.
Let $I$ be an interval which contains $\sup B$.
If $|A+B|\le |A|+|B|+\delta$, then
\begin{equation} |A+(B\setminus I)|\le |A|+|B\setminus I|+\delta.\end{equation} 
\end{lemma}

\begin{proof}
By independently translating $A,B$ we may assume without loss of generality
that $\sup A=0$ and that $0$ is the left endpoint of $I$.
Then $A+(B\setminus I)\subset (-\infty,0]$ while
$\set{0}+(B\cap I)\subset[0,\infty)$.
Since $0\in A$, 
\[A+B\supset [A+(B\setminus I)]\cup [\set{0}+(B\cap I)];\]
moreover, this is a disjoint union except possibly for the point $0$.
Therefore
\begin{align*} |A+B|&\ge |A+(B\setminus I)| + |\set{0}+(B\cap I)|
\\
&= |A+(B\setminus I)| + |B\cap I|.  \end{align*}
Therefore
\begin{align*} |A+(B\setminus I)| &\le |A+B|- |B\cap I|
\\
&< |A|+|B|+\delta-|B\cap I| 
\\ &= |A|+|B\setminus I|+\delta.  \end{align*}
\end{proof}

\begin{lemma} \label{lemma:1Dlocalize}
Let $A,B\subset\reals^1$ be compact sets.  Let $I,J$ be intervals. 
If $|A+B|< |A|+|B|+\delta$, then \[|(A\cap I)+(B\cap J)| < |A\cap I|+|B\cap J|+\delta.\]  
\end{lemma}
This is proved by applying the preceding lemma up to four times, and exploiting
the symmetry $x\mapsto -x$ of $\reals$. \qed

In $\reals^d=\reals^1\times\reals^{d-1}$ for $d>1$, the following partial analogue holds.
\begin{lemma} \label{lemma:pre2Dlocalization}
Let $A,B\subset\reals^d$ be compact. Let $b\in\reals^1$ and set 
$B^-=B\cap \pi^{-1}((-\infty,b])$ and $B^+=B\setminus B^-$.
Let $a=\sup\set{x\in\reals^1: A_x\ne\emptyset}$, and
suppose that $|A_a|\ge \sup_{x\ge b} |B_x|$.
If $|A|,|B|=1+O(\delta)$ and if 
\[|t A+(1-t) B|<|A|^t|B|^{1-t}+\delta\]
then
\[|tA+(1-t) B^-|<|A|^t|B^-|^{1-t}+O(\delta)+O(|B^+|^2).\]
\end{lemma}
When this result is invoked below, 
it will be crucial that the final term on the right is $o(|B^+|)$, rather than merely $O(|B^+|)$.

\begin{proof}
By translating $A$ we may reduce to the case in which $a=b=0$.  Then 
\begin{align*} t A + (1-t) B^- &\subset \pi^{-1}((-\infty,0]), 
\\ t (\set{0}\times A_0)+(1-t) B^+&\subset\pi^{-1}([0,\infty)).  \end{align*}
Thus \[|t A+(1-t) B|\ge  |t A+(1-t) B^-| +  |t (\set{0}\times A_0)+(1-t) B^+|.\]

We claim that
\begin{equation} \label{eq:slicebound1} |t (\set{0}\times A_0)+(1-t) B^+|\ge (1-t) |B^+|.\end{equation}
Indeed, for any $x\ge 0$, 
\[\set{y\in\reals^{d-1}: (x,y)\in t (\set{0}\times A_0)+(1-t) B^+} =t A_0+(1-t) B_{x/(1-t)}.\] 
By the Brunn-Minkowski inequality for $\reals^{d-1}$ in the form \eqref{eq:BMmult},
\[|t A_0+(1-t) B_{x/(1-t)}|\ge  |A_0|^t |B_{x/(1-t)}|^{1-t} \ge |B_{x/(1-t)}|\]
by the hypothesis $|A_0|\ge|B_{x/(1-t)}|$. Upon integrating over all $x\ge 0$, we obtain 
\eqref{eq:slicebound1} via Fubini's theorem.

Therefore 
\[|A|^t|B|^{1-t}+\delta>|tA+(1-t) B| \ge |t A+(1-t) B^-| + (1-t)|B^+|\]
and consequently
\[ |t A+(1-t) B^-| <|A|^t|B|^{1-t}+\delta - (1-t)|B^+| = 1+O(\delta)-(1-t)|B^+|.  \]
On the other hand, if $|B^+|$ is small then
\begin{align*}
|A|^t|B^-|^{1-t} &= |A|^t(|B|-|B^+|)^{1-t} 
\\
&= (1+O(\delta))^t(1-|B^+|+O(\delta))^{1-t}
\\
&= (1+O(\delta)) (1-(1-t)|B^+| + O(\delta)+O(|B^+|^2))
\\
&= 1-(1-t)|B^+| + O(\delta) + O(|B^+|^2).
\end{align*}
Therefore
\[ |tA+(1-t)B^-|< |A|^t|B^-|^{1-t} + O(\delta)+O(|B^+|^2). \]
\end{proof}

Two applications of the preceding lemma give the following extension.
\begin{lemma} \label{lemma:localization}
Let $A,B\subset\reals^d$ be compact,
let $[c_0,c_1]\subset\reals^1$ be a compact interval,
and let $\tilde B=B\cap ([c_0,c_1]\times\reals^{d-1})$.
Let 
\[a_0=\sup\set{x\in\reals: A_x\ne\emptyset}\ \text{ and }\ a_1=\inf\set{x\in\reals: A_x\ne\emptyset}.\] 
Suppose that \[|A_{a_1}|\ge \sup_{x\ge c_1} |B_x| \text{ and } |A_{a_0}|\ge \sup_{x\le c_0} |B_x|.\]
If $|A|,|B|=1+O(\delta)$ and $|t A+(1-t) B|<1+O(\delta)$
then \[|t A+(1-t) \tilde B|<
|A|^t|\tilde B|^{1-t}
+O(\delta)+O(|B\setminus\tilde B|^2).\]
\end{lemma}

\section{Step One: Vertical normalization}

Let $A,B\subset\reals^d$ be nonempty compact sets.  Denote the associated sum set by $S=t A+(1-t)B$.  
For $x\in \reals^1$, $A_x=\set{y\in\reals^1: (x,y)\in A}$; $B_x,S_x$ are the corresponding sets.
Since $A,B,S$ are compact, the functions  $|A_x|,|B_x|,|S_x|$ 
are bounded, are upper semicontinuous, and have compact supports.  Define 
$\norm{A}_\infty$ to be the maximum of the function $x\mapsto |A_x|$,
and define the corresponding sets  $\norm{B}_\infty,\norm{S}_\infty$ in the same way.

\begin{lemma} \label{lemma:verticalsupratio}
Suppose that $A,B\subset\reals^d$ satisfy the hypotheses \eqref{hypos}.
Then
\begin{equation} \frac{\norm{A}_\infty}{\norm{S}_\infty} = 1+O(\delta^{1/2}).  \end{equation}
The same holds for $\norm{B}_\infty/\norm{S}_\infty$.  \end{lemma}

\begin{proof}
By the Brunn-Minkowski inequality \eqref{eq:BMmult} for subsets of $\reals^{d-1}$,
\begin{equation} |S_{t x'+(1-t)x''}| \ge  |A_{x'}|^t |B_{x''}|^{1-t} \end{equation}
for all $x',x''$ such that $A_{x'},B_{x''}$ are nonempty.  Therefore  
\begin{equation} \label{sublevel}
\set{x: |S_x|>\lambda}\supset t\set{x': |A_{x'}|>\gamma\lambda} +(1-t)\set{x'': |B_{x''}|>\tilde \gamma\lambda}
\end{equation}
whenever both of the sets on the right are nonempty and $\gamma^t\tilde\gamma^{1-t}=1$.
Indeed for any such $x',x''$, 
\begin{equation}
|A_{x'}|^t |B_{x''}|^{1-t} > \gamma^t\lambda^t\tilde\gamma^{1-t}\lambda^{1-t}=\lambda.  \end{equation}
Therefore by the Brunn-Minkowski inequality for $\reals^1$,
\begin{equation}  \label{eq:BM:superlevelsets}
|\set{x: |S_x|>\lambda}|\ge t|\set{x: |A_x|>\gamma\lambda}| +(1-t)|\set{x: |B_x|>\tilde \gamma\lambda}|.
\end{equation}

The one-dimensional Brunn-Minkowski inequality is used here in its additive form.
The following steps rely on its taking this simple form for $d=1$, in contrast to the
general form $|tU+(1-t)V|\ge (t|U|^{1/d}+(1-t)|V|^{1/d})^d$ in dimension $d$.
It is this step which leads us to factor $\reals^d$ as $\reals^1\times\reals^{d-1}$.
A subsequent step will require a simultaneous factorization as $\reals^{d-1}\times\reals^1$,
forcing the restriction $d=2$ in our hypotheses.

Define 
\[\gamma=\big(\norm{A}_\infty/\norm{B}_\infty\big)^{t}
\text{ and }  \tilde\gamma = \big(\norm{B}_\infty/\norm{A}_\infty\big)^{1-t} = \gamma^{-(1-t)/t}\]
so that $\gamma^t\tilde\gamma^{1-t}=1$ and $\gamma^{-1}\norm{A}_\infty = \tilde\gamma^{-1}\norm{B}_\infty$.
Then \eqref{sublevel} applies for all $\lambda < \gamma^{-1}\norm{A}_\infty = \tilde\gamma^{-1}\norm{B}_\infty$.

By the hypothesis $|tA+(1-t)B|<1+\delta$,
\begin{align*}
1+\delta& >|S|
\\
&=\int_0^\infty |\set{x: |S_x|>\lambda}|\,d\lambda
\\
&\ge \int_0^{\gamma^{-1}\norm{A}_\infty} |\set{x: |S_x|>\lambda}|\,d\lambda
\\
&\ge t\int_0^{\gamma^{-1}\norm{A}_\infty} |\set{x: |A_x|>\gamma \lambda}|\,d\lambda
+
(1-t)\int_0^{\gamma^{-1}\norm{A}_\infty} |\set{x: |B_x|>\tilde \gamma \lambda}|\,d\lambda
\\
&=t\gamma^{-1} \int_0^{\norm{A}_\infty} |\set{x: |A_x|>\lambda}|\,d\lambda
+
(1-t)\tilde\gamma^{-1} \int_0^{\norm{B}_\infty} |\set{x: |B_x|> \lambda}|\,d\lambda
\\
&=t\gamma^{-1}|A|+(1-t)\tilde\gamma^{-1} |B|
\\ &=t(1+O(\delta))\gamma^{-1}+(1-t)(1+O(\delta))\tilde\gamma^{-1}
\end{align*}
so
\[ \big(t\gamma^{-1}+(1-t)\gamma^{t/(1-t)} \big) \le 1+\delta  +O(\delta)\big(\gamma^{-1}+\gamma^{t/(1-t)}\big).  \]
Since $t$ is assumed to lie in a fixed compact subset of $(0,1)$,
this forces \[\gamma=1+O(\delta^{1/2})\] provided that $\delta$ is sufficiently small. 
\end{proof}

\section{Step Two: Upper bounds for horizontal projections}
For $E\subset\reals^d=\reals^1\times\reals^{d-1}$, denote by $\pi(E)$ the projection
of $E$ onto the factor $\reals^1$.
By Lemma~\ref{lemma:verticalsupratio}, it is possible to make
a change of variables of $\reals^1\times\reals^{d-1}$ of the form
$(x_1,x')\mapsto(r_1x_1,r_2x')$ with $r_1r_2^{d-1}=1$ to ensure that
\begin{equation} \label{eq:verticalsupsnormalized} \norm{A}_\infty=1+O(\delta^{1/2}) \ \text{ and } \ 
\norm{B}_\infty =1+O(\delta^{1/2}). \end{equation}
We assume this henceforth.  

\begin{lemma}\label{lemma:piAbound} 
There exists $M<\infty$ such that for all sufficiently small $\delta>0$, whenever
$A,B,t$ satisfy \eqref{hypos} and the supplementary normalization \eqref{eq:verticalsupsnormalized},
\begin{equation} M^{-1}\le \min(|\pi(A)|,|\pi(B)|)\le \max(|\pi(A)|,|\pi(B)|)\le M. \end{equation} \end{lemma}

\begin{proof}
Since $\norm{A}_\infty=1+O(\delta^{1/2})$, $A$ must contain some set of the form $\set{x}\times E$
where $E\subset\reals^{d-1}$ is a compact set of measure $1+O(\delta^{1/2})$ and $x\in\reals^1$.
Therefore $\tatb$ contains $ (\set{tx}\times tE)+(1-t)B$, and consequently
\[ 1+\delta > |\tatb|\ge (1-t) |\pi(B)|\cdot t^{d-1} |E| \ge (1+O(\delta^{1/2}))t^{d-1}(1-t)|\pi(B)|.  \]
For sufficiently small $\delta$, this gives the upper bound for $|\pi(B)|$. $A$ is treated in the same way.

On the other hand, the lower bounds for the measures of the projections
are immediate consequences of the upper bounds \eqref{eq:verticalsupsnormalized}.
\end{proof}

\section{Step Three: Horizontal structure}
For $s\ge 0$ define 
\begin{align*}
&\scripta_s=\set{x\in\reals^1: |A_x|>s}
\\
&\scriptb_s=\set{x\in\reals^1: |B_x|>s}  
\\
&\scripts_s=\set{x\in\reals^1: |S_x|>s}
\end{align*}
where $S = tA+(1-t)B$.
By making a measure-preserving affine change of variables, we have reduced matters to the case
in which by Lemma~\ref{lemma:piAbound}, $|\scripta_s|\le M$ and $|\scriptb_s|\le M$ for all $s>0$,
where $M<\infty$ is some absolute constant.

\begin{lemma} 
Let $A,B,t$ satisfy the hypotheses \eqref{hypos} and \eqref{eq:verticalsupsnormalized}.  Then
\begin{equation} \int_{\min(\norm{A}_\infty,\norm{B}_\infty)}^\infty (|\scripta_s|+|\scriptb_s|)\,ds = O(\delta^{1/2}). 
\end{equation} \end{lemma}

\begin{proof}
The integrand is majorized by $2M$, and is supported on the small interval $[1-O(\delta^{1/2}),1+O(\delta^{1/2})]$.
\end{proof}

Define
\begin{equation}\phi(s)=|\scripts_s|-t|\scripta_{s}|- (1-t) |\scriptb_{s}|.\end{equation}

\begin{lemma}
Let $A,B,t$ satisfy the hypotheses \eqref{hypos} and the normalization \eqref{eq:verticalsupsnormalized}. 
If $s< \min(\norm{A}_\infty,\norm{B}_\infty)$ then $\phi(s)\ge 0$.  Moreover, 
\begin{equation} \int_0^{\min(\norm{A}_\infty,\norm{B}_\infty)} \phi(s)\,ds =O(\delta^{1/2}).  \end{equation}
\end{lemma}

\begin{proof}
Whenever $\scripta_s,\scriptb_s$ are both nonempty, 
$\scripts_s\supset t \scripta_{s}+ (1-t)\scriptb_{s}$ by \eqref{eq:BM:superlevelsets}.
By this inclusion and the one-dimensional Brunn-Minkowski inequality, 
\[\phi(s)\ge |t\scripta_s+(1-t)\scriptb_s|-t|\scripta_s|-(1-t)|\scriptb_s|\ge 0 \]
whenever $s< \min(\norm{A}_\infty,\norm{B}_\infty)$.
Furthermore
\begin{align*}
\int_0^{\min(\norm{A}_\infty,\norm{B}_\infty)} 
\phi(s)\,ds &
\\
&\le |\tatb|-t|A|-(1-t)|B| 
\\
&\qquad + \int_{\min(\norm{A}_\infty,\norm{B}_\infty)}^\infty (t|\scripta_s|+(1-t)|\scriptb_s|)\,ds 
\\
&\le \delta + \int_{\min(\norm{A}_\infty,\norm{B}_\infty)}^{\max(\norm{A}_\infty,\norm{B}_\infty)}
M\,ds
\\
&<\delta+O(\delta^{1/2}).  \end{align*}
\end{proof}

\begin{lemma} \label{lemma:gotintervals}
Let $A,B,t$ satisfy the hypotheses \eqref{hypos} and the normalization \eqref{eq:verticalsupsnormalized}. 
There exists $C_0<\infty$ with the following property.
Suppose that $0<2\eta\le \sigma<\min(\norm{A}_\infty,\norm{B}_\infty)$, and that
\[ \min(t |\scripta_\sigma|,\,(1-t) |\scriptb_\sigma|)\ge C_0\delta^{1/2}\eta^{-1}.  \]
Then there exist $\tau\in[\sigma-\eta,\sigma]$ and intervals $I,J\subset\reals$ such that 
\begin{alignat*}{2}
&\scripta_\tau\subset I,\qquad &&|I|<|\scripta_\tau|+O(\delta^{1/2}\eta^{-1}),
\\
&\scriptb_\tau\subset J, &&|J|<|\scriptb_\tau|+O(\delta^{1/2}\eta^{-1}).
\end{alignat*}
\end{lemma}

\begin{proof}
\[ \eta^{-1}\int_{\sigma-\eta}^\sigma \phi(s)\,ds
\le \eta^{-1}\int_0^{\min(\norm{A}_\infty,\norm{B}_\infty)} \phi(s)\,ds = \eta^{-1}O(\delta^{1/2}), \]
so there exists $s\in[\sigma-\eta,\sigma]$ such that $\phi(s)=O(\delta^{1/2}\eta^{-1})$, which is to say,
\[|t\scripta_s+(1-t)\scriptb_s|<t|\scripta_s|+(1-t)|\scriptb_s|+ O(\delta^{1/2}\eta^{-1}).\]
Since $s\le\sigma$, $\scripta_s\supset \scripta_\sigma$ and therefore 
\[t|\scripta_s|\ge t|\scripta_\sigma|\ge C_0\delta^{1/2}\eta^{-1}.\]
Likewise $(1-t)|\scriptb_s|\ge C_0\delta^{1/2}\eta^{-1}$.
Thus if $C_0$ is sufficiently large, then 
\[|t\scripta_s+(1-t)\scriptb_s|<t|\scripta_s|+(1-t)|\scriptb_s|+\min(|t\scripta_s|,|(1-t)\scriptb_s|);\]
therefore the hypothesis of Proposition~\ref{prop:nearlyintervals} is satisfied. 
That result guarantees the existence of the required intervals $I,J$.
\end{proof}

If $\sigma$ is small then $|\scripta_\sigma|$ cannot be small, for 
\[1-\delta \le |A|=\int_0^{\norm{A}_\infty}|\scripta_s|\,ds
\le M\sigma + \int_\sigma^{1+O(\delta^{1/2})} |\scripta_\sigma|\,ds
\le  M\sigma+2|\scripta_\sigma|\] 
for all sufficiently small $\delta$.
Thus $2|\scripta_\sigma|\ge 2-M\sigma$. If $\sigma \le \tfrac12 M^{-1}$ and $\delta$ is sufficiently
small, it follows that $|\scripta_\sigma|\ge\tfrac14$.

Therefore we may apply Lemma~\ref{lemma:gotintervals} with
$\eta=\delta^{1/6}$ to conclude the following.
\begin{lemma} \label{lemma:gotintervals2}
Let $A,B,t$ satisfy the hypotheses \eqref{hypos} and the normalization \eqref{eq:verticalsupsnormalized}. 
There exist $s_0\in[\delta^{1/6},2\delta^{1/6}]$
and compact intervals $\scripti,\scriptj\subset\reals^1$ such that
\begin{equation}\label{bound23}
\begin{alignedat}{2}
&\scripta_{s_0}\subset \scripti,\qquad &&|\scripti|<|\scripta_{s_0}|+O(\delta^{1/3}),
\\
&\scriptb_{s_0}\subset \scriptj, &&|\scriptj|<|\scriptb_{s_0}|+O(\delta^{1/3}).
\end{alignedat}
\end{equation}
Moreover,
$|A_y|\ge s_0$ for both endpoints $y$ of $\scripti$,
and likewise
$|B_y|\ge s_0$ for both endpoints $y$ of $\scriptj$.
Finally,
\begin{equation}\label{bound16}
|A\cap \pi^{-1}(\reals\setminus\scripti)| +|B\cap \pi^{-1}(\reals\setminus\scriptj)| =O(\delta^{1/6}).
\end{equation}
\end{lemma}

\begin{proof}
The existence of $\scripti,\scriptj$ satisfying \eqref{bound23} has already been shown.
These conclusions remain valid if $\scripti$ is replaced by its intersection with the
compact interval $[\inf(\scripta_{s_0}),\sup(\scripta_{s_0})]$. 
If $(y_\nu)$ is a sequence of points of $\scripta_{s_0}$ which converges to 
$y=\sup(\scripta_{s_0})$, then from the compactness of $A$ it follows that 
\[|A_y|\ge\limsup_{\nu\to\infty} |A_{y_\nu}|\ge s_0.\] 
Likewise $|A_y|\ge s_0$ for $y=\inf(\scripta_{s_0})$.
The same reasoning applies to $\scriptb_{s_0}$.

Only the final conclusion remains to be verified.
For $x\in\reals \setminus \scripta_{s_0}$, $|A_x|\le s_0$. Therefore
\begin{equation} %\label{eq:piinversebound}
|A\cap \pi^{-1}(\reals\setminus\scripti)| \le \int_{\pi(A)} s_0\,dx \le Ms_0\le 2M\delta^{1/6}.
\end{equation}
The same reasoning applies to $B\cap \pi^{-1}(\reals\setminus\scriptj)$.
\end{proof}

\section{Step Four: Near equality of horizontal projections}
By independently translating the sets $A,B$ with respect to the $\reals^1$ coordinate
we may assume without loss of generality henceforth that
\[\scripti=[0,a] \text{ and } \scriptj=[0,b] \text{ for some $a,b\in[0,2M]$.}\]
$a,b$ are also bounded below, since $|\scripta_{s_0}|,|\scriptb_{s_0}|$ are bounded below.
We assume henceforth that $A,B$ have been translated so that these conclusions hold.  
This is merely a matter of notational convenience, since it does not affect the measure
of $tA+(1-t)B$.

Now for the first time we assume that $A,B\subset\reals^2=\reals\times\reals$.
\begin{lemma} Let $d=2$. Then $a-b=O(\delta^{1/12})$. \end{lemma}

\begin{proof}
Continue to let $S=tA+(1-t)B$, and set \[c=ta+(1-t)b.\]
For $x\in [0,c]=t[0,a]+(1-t)[0,b]$, write
\begin{equation} x = tx'+(1-t)x'' \text{ where } x' = x'(x)=\frac{a}{c}x 
\text{ and }  x''=x''(x)=\frac{b}{c}x.  \end{equation}
Thus $x'\in [0,a]$ and $x''\in [0,b]$.

Let \[\scriptd=\set{x\in [0,c]: |A_{x'}|>0 \text{ and } |B_{x''}|>0}.\]
Since
\[|\set{x\in[0,a]: |A_x|=0}|\le |\set{x\in[0,a]: |A_x|\le s_0}| =|\scripti\setminus \scripta_{s_0}| =O(\delta^{1/3})\]
and likewise
$|\set{x\in[0,b]: |B_x|=0}|=O(\delta^{1/3})$,
\begin{equation} |[0,c]\setminus\scriptd|=O(\delta^{1/3}).\end{equation}
Since \[c=ta+(1-t)b\le t2M+(1-t)2M\le 2M\] and 
\[\sup_{x'}|A_{x'}|,\sup_{x''}|B_{x''}|\le 1+O(\delta^{1/2}) =O(1),\] 
it follows that
\[\int_{[0,c]\setminus\scriptd} (|A_{x'}|+|B_{x''}|)\,dx 
\le O(1)\big|[0,c]\setminus\scriptd\big|=O(\delta^{1/3}).\]

Now $S_x,A_{x'},B_{x''}\subset\reals^1$ since $d=2$,
and $tA_{x'}+(1-t)B_{x''}\subset S_x$. Therefore
by the one-dimensional Brunn-Minkowski inequality, for any $x\in\scriptd$,
\begin{equation}|S_x|\ge t|A_{x'(x)}|+(1-t)|B_{x''(x)}|.\end{equation}

Integrating this inequality with respect to $x\in\scriptd$ gives
\begin{align*}
1+\delta>|S|&=\int_{\reals^1}|S_x|\,dx
\\ &\ge \int_{\scriptd} \big( t |A_{x'(x)}| + (1-t) |B_{x''(x)}|\big)\,dx
\\ &\ge \int_{0}^{c} \big( t |A_{x'(x)}| + (1-t) |B_{x''(x)}|\big)\,dx\ -O(\delta^{1/3}).
\end{align*}
Split the integral as a sum of two integrals, and change variables in each to obtain
\begin{align*}
1+\delta &>tca^{-1} \int_0^a |A_r|\,dr +(1-t)cb^{-1}\int_0^b |B_r|\,dr -O(\delta^{1/3})
\\ &\ge tca^{-1} |A| +(1-t)cb^{-1} |B| -O(\delta^{1/3})
\\ & \qquad -tca^{-1}|A\setminus\pi^{-1}([0,a])| -(1-t)cb^{-1}|B\setminus\pi^{-1}([0,b])|
\\
&\ge tca^{-1}+(1-t)cb^{-1}-O(\delta^{1/3})-O(\delta^{1/6})(ca^{-1}+cb^{-1});
\end{align*}
the second line followed from \eqref{bound16}.
We have shown that
\begin{equation*} tca^{-1}+(1-t)cb^{-1}<1+O(\delta^{1/6})(ca^{-1}+cb^{-1}) \end{equation*}
where $c=ta+(1-t)b$. For $t$ in any fixed compact subset of $(0,1)$,
this inequality forces $a/b=1+O(\delta^{1/12})$.
Since $a,b$ lie in a fixed compact subset of $(0,\infty)$, that is equivalent to $a-b=O(\delta^{1/12})$.
\end{proof}

\section{Step Five: Near equality for truncated sets} 

Assume with no loss of generality that $a\ge b$.
Introduce the truncated sets 
\[A^\sharp=A\cap ([0,a]\times\reals) \text{ and }B^\sharp=B\cap ([0,a]\times\reals).\]
These sets are still compact. 

\begin{lemma}
\begin{equation} \label{happilysharp}
\Big|\,|tA^\sharp+(1-t)B^\sharp| - t|A^\sharp|-(1-t)|B^\sharp|\,\Big| =O(\delta^{1/3}).
\end{equation}
\end{lemma}
This conclusion may look odd, since $A^\sharp,B^\sharp$ are subsets of $\reals^2$, 
yet this is the type of expression which appears in the Brunn-Minkowski inequality for $\reals^1$. 
But since $A^\sharp,B^\sharp$ have nearly equal measures, $t|A^\sharp|+(1-t)|B^\sharp|$
is nearly equal to $(t|A^\sharp|^{1/2}+(1-t)|B^\sharp|^{1/2})^2$.

\begin{proof}
$A^\sharp,B^\sharp$ satisfy
$|A^\sharp|,|B^\sharp|=1+O(\delta^{1/6})$ by \eqref{bound16}.
Moreover, by the construction of $a$ and the localization lemma, Lemma~\ref{lemma:localization}, 
\begin{equation}\begin{aligned} 
|t A^\sharp+(1-t) B^\sharp|
& < |A^\sharp|^{t}|B^\sharp|^{1-t}+O(\delta) +O(|B\setminus B^\sharp|^2)
\\ &= |A^\sharp|^{t}|B^\sharp|^{1-t}+O(\delta) +O(\delta^{1/6})^2
\\ &=|A^\sharp|^t|B^\sharp|^{1-t}+O(\delta^{1/3}).
\label{happilylocalized} 
\end{aligned} \end{equation} 

For any $\alpha,\beta>0$,
\[ \alpha^t\beta^{1-t} = t\alpha+(1-t)\beta + O\Big((\frac\alpha\beta-1 )^2\Big) \max(\alpha,\beta).  \]
According to Lemma~\ref{lemma:goodratio}, the inequality 
$ |t A^\sharp+(1-t) B^\sharp|< |A^\sharp|^t|B^\sharp|^{1-t}+O(\delta^{1/3})$ implies that
\begin{equation} 
\frac{|A^\sharp|}{|B^\sharp|} = 1+O(\delta^{1/6}).
\end{equation}
Therefore
\begin{align*}
|A^\sharp|^t|B^\sharp|^{1-t} &= t|A^\sharp|+(1-t)|B^\sharp| + O\big((\delta^{1/6})^2\big) 
\max(|A^\sharp|,|B^\sharp|)) 
\\ &= t|A^\sharp|+(1-t)|B^\sharp| + O(\delta^{1/3}).  \end{align*}
\end{proof}

One can use the bounds
$|A\setminus A^\sharp|=O(\delta^{1/6})$ and $|B\setminus B^\sharp|=O(\delta^{1/6})$
more directly, without invoking Lemma~\ref{lemma:localization}, to deduce that 
$\Big|\,|tA^\sharp+(1-t)B^\sharp| - t|A^\sharp|-(1-t)|B^\sharp|\,\Big| =O(\delta^{1/6})$.
However, this power of $\delta$ does not suffice for subsequent steps of the analysis.

\section{Step Six: Vertical structure}

Throughout this section, $c=ta+(1-t)b$ where $\scripti=[0,a]$ and $\scriptj=[0,b]$
satisfy the conclusions of Lemma~\ref{lemma:gotintervals2}.
For $x\in[0,c]$, we systematically write
\begin{equation}
x = tx'+(1-t)x'' \ \text{ where } \ x'=x'(x)=ac^{-1}x \ \text{ and } \ x''=x''(x)=bc^{-1}x.
\end{equation}
We have shown that the quantities $a/c$, $b/c$ are both equal to $1+O(\delta^{1/12})$.

The goal of step 6 is to achieve:
\begin{proposition} \label{prop:intervalsKL}
Assume that the dimension $d$ equals $2$. 
Let $t\in\Lambda$.
Let $A,B\subset\reals^2$  be compact sets
which satisfy the hypotheses \eqref{hypos} and the normalization \eqref{eq:verticalsupsnormalized}, 
and are appropriately translated.
Let $a,b,c>0$ be as defined above.
Then there exists a subset $\omega^\ddagger\subset[0,c]$ satisfying
\begin{equation} |[0,c]\setminus \omega^\ddagger| =O(\delta^{1/12})  \end{equation}
such that for each $x\in\omega^\ddagger$ there exist compact intervals $I_{x'(x)},J_{x''(x)}\subset\reals^{1}$
such that 
\begin{alignat}{2}
&A_{x'(x)}\subset I_{x'(x)},\ \  &&|I_{x'(x)}|\le|A_{x'(x)}|+\delta^{1/4},
\\
&B_{x''(x)}\subset J_{x''(x)}, &&|J_{x''(x)}|\le |B_{x''(x)}|+\delta^{1/4}.
\end{alignat}
\end{proposition}
The various exponents which appear in our inequalities should by no means be regarded as definitive.

\begin{proof}
Throughout this proof, $x'\equiv x'(x)$ and $x''\equiv x''(x)$.
Define \[\omega=\set{x\in [0,c]: A_{x'}\ne\emptyset \text{ and } B_{x''}\ne\emptyset}.\]
By \eqref{bound23}, $|[0,c]\setminus\omega|= O(\delta^{1/3})$.
Now
\begin{align*}
|t A^\sharp+(1-t) B^\sharp| 
&\ge \int_{\reals^1} |t A_{x'}+(1-t) B_{x''}|\,dx
\\ &\ge \int_{\omega} |t A_{x'}+(1-t) B_{x''}|\,dx
\\ &= 
\int_{\omega} \Big(|t A_{x'}+(1-t) B_{x''}|-t|A_{x'}|-(1-t)|B_{x''}|\Big)\,dx
\\
&\qquad +ca^{-1}t|A^\sharp|+cb^{-1}(1-t)|B^\sharp| 
-\int_{[0,c]\setminus\omega}(t|A_{x'}|+(1-t)|B_{x''}|)\,dx. 
\end{align*}
Now
\[ \int_{[0,c]\setminus\omega} (t|A_{x'}|+(1-t)|B_{x''}|)\,dx = O(|[0,a]\setminus\omega) = O(\delta^{1/3}). \]
Also since $\frac{a}{c}+\frac{b}{c}= 1$, $\frac{c}{a}=O(1)$, $\frac{c}{b}=O(1)$, 
and $|A^\sharp|-|B^\sharp|=O(\delta^{1/3})$,
\begin{align*}
ca^{-1}t|A^\sharp|+cb^{-1}(1-t)|B^\sharp| 
&\ge ca^{-1}t|A^\sharp|+cb^{-1}(1-t)|A^\sharp| +O(\delta^{1/3})
\\
&\ge |A^\sharp|+O(\delta^{1/3})
\\
&\ge t|A^\sharp|+(1-t)|B^\sharp| + O(\delta^{1/3}).
\end{align*}
% ??? fixed April 24? Is all this correct?
It follows that
\begin{equation} \label{eq:anotherintegralbound}
\int_{\omega} \Big(|t A_{x'}+(1-t) B_{x''}|-t|A_{x'}|-(1-t)|B_{x''}|\Big)\,dx =O(\delta^{1/3}).
\end{equation}

By the Brunn-Minkowski inequality applied to $A_{x'},B_{x''}\subset\reals^1$,
\begin{equation} |t A_{x'(x)}+(1-t) B_{x''(x)}|-t|A_{x'(x)}|-(1-t)|B_{x''(x)}|\ \ge 0 \ \ 
\text{ whenever $x\in\omega$.} \end{equation}
Therefore by Chebyshev's inequality and \eqref{eq:anotherintegralbound}, for any $\lambda>0$,
\[|\set{x\in\omega: |tA_{x'(x)}+(1-t)B_{x''(x)}|-t|A_{x'(x)}|-(1-t)|B_{x''(x)}|>\lambda}| 
= O(\lambda^{-1}\delta^{1/3}).\]

The subset $\omega^\dagger\subset\omega$ defined to be
\begin{equation}\label{eq:omegadagger}
\omega^\dagger=\set{x\in[0,c]: |A_{x'(x)}|\ge \delta^{1/6} \text{ and } |B_{x''(x)}|\ge\delta^{1/6}}\end{equation} 
satisfies 
\begin{equation}\big|[0,c]\setminus\omega^\dagger\big| = O(\delta^{1/6})\end{equation}
by \eqref{bound23}, since $s_0\ge \delta^{1/6}$, $|A_y|>s_0$ for $y\in \scripta_{s_0}$,
and $|[0,a]\setminus \scripta_{s_0}|=O(\delta^{1/3})$, with corresponding statements for the set $B$.

% ??? April 22: amplify
Choose $\lambda = \delta^{1/4}$, and define
\begin{equation}\omega^\ddagger = 
\set{x\in \omega^\dagger: |tA_{x'(x)}+(1-t)B_{x''(x)}|-t|A_{x'(x)}|-(1-t)|B_{x''(x)}|\le \delta^{1/4}}.\end{equation}
Then we have shown that
\begin{equation}|[0,c]\setminus\omega^\ddagger|=O(\delta^{1/6}+\delta^{1/12}) = O(\delta^{1/12}).\end{equation}

For each $x\in\omega^\ddagger$, the pair of sets $A_{x'(x)},B_{x''(x)}\subset\reals^{1}$ gives rise to near equality
in the one-dimensional Brunn-Minkowski inequality.    
Moreover, $\min(|A_{x'}|,|B_{x''}|)\ge \delta^{1/6}$, while 
\[|tA_{x'}+(1-t)B_{x''}|-t|A_{x'}|-(1-t)|B_{x''}|\le \delta^{1/4}\ll\delta^{1/6}.\]
Therefore by Proposition~\ref{prop:nearlyintervals}, 
for each $x\in\omega^\ddagger$ there exist intervals $I_{x'},J_{x''}\subset\reals$
such that $A_{x'}\subset I_{x'}$, $B_{x''}\subset J_{x''}$, 
$|I_{x'}|<|A_{x'}|+O(\delta^{1/4})$, and $|J_{x''}|<|B_{x''}|+O(\delta^{1/4})$.
\end{proof}

It is possible to choose the families of intervals $I_x,J_x$ in Proposition~\ref{prop:intervalsKL}
so that the set of all $(x,y)\in [0,a]\times\reals^{1}$ such that 
$ca^{-1}x\in\omega^\ddagger \text{ and } y\in I_x$ is Lebesgue measurable, and likewise 
$\set{(x,y)\in[0,b]\times\reals^{1}: cb^{-1} x\in\omega^\ddagger \text{ and } y\in J_x}$ is measurable.
Since $|A_x|,|B_x|\le 1+O(\delta^{1/2})$ for all $x$, 
\[ |I_{x'}|\le 2 \ \text{ and } \  |J_{x''}|\le 2 \ \text{ for all $x\in\omega^\ddagger$}, \]
provided that $\delta$ is sufficiently small.

\section{Step Seven: Two-dimensional affine structure}

Throughout the remainder of the paper we assume that the dimension is $d=2$.
Define $\omega^\ddagger_A=\set{x'(x): x\in\omega^\ddagger}$
and $\omega^\ddagger_B=\set{x''(x): x\in\omega^\ddagger}$.
For each $x\in\omega^\ddagger_A$ define $\varphi(x)$ to be the center of $I_{x'}$;
for $x\in\omega^\ddagger_B$, $\psi(x)$ is likewise the center of $J_{x}$.
Extend these two functions in some arbitrary manner to obtain measurable functions,
still denoted $\varphi,\psi$, with domains equal to  $[0,a]$ and to $[0,b]$, respectively.  Define 
$\omega^\times$ to be the set of all $(x',x'')\in\omega^\ddagger_A\times\omega^\ddagger_B$  such that either
$tx'+(1-t)x''\notin\omega^\ddagger$, or $tx'+(1-t)x'' \in\omega^\ddagger $ and
\[|\big(t\varphi(x')+(1-t)\psi(x'')\big) - \big(t\varphi(tx'+(1-t)x'')+(1-t)\psi(tx'+(1-t)x'') \big)|\ge 2.\]

In the following discussion, $y\mapsto x'(y)$ and $y\mapsto x''(y)$
are the same mappings $y\mapsto ca^{-1}y$, $y\mapsto cb^{-1}y$ as above,
while $x',x''$ will denote general points in $[0,a]$, $[0,b]$ respectively.

\begin{lemma} $|\omega^\times|=O(\delta^{1/6})$.  \end{lemma}

\begin{proof} 
Let 
\[ E=\set{ (x',x'')\in\omega^\times: tx'+(1-t)x''\notin\omega^\ddagger }. \]
Since $|\omega^\dagger|\asymp 1$
and the mappings $x\mapsto x'$ and $x\mapsto x''$ are affine functions with
derivatives $1+O(\delta^{1/12})\asymp 1$ which map $[0,c]\setminus\omega^\ddagger$
bijectively to $[0,a]\setminus\omega^\ddagger_A$
and to $[0,b]\setminus\omega^\ddagger_B$ respectively,
\[ |E|=O(|[0,a]\setminus\omega^\ddagger_A|) + O(|[0,b]\setminus\omega^\ddagger_B|) =O(\delta^{1/12}).\]

Consider any $y\in\omega^\ddagger$ for which there exists $(x',x'')\in\omega^\times$
satisfying $tx'+(1-t)x''=y$, with the additional property that 
\begin{equation} \label{eq:farcenters}
\big|\big(t\varphi(x')+(1-t)\psi(x'')\big) - \big(t\varphi(x'(y))+(1-t)\psi(x''(y)) \big)\big|\  >\  2.
\end{equation}
The center of an algebraic sum of two intervals equals the algebraic sum of their centers.
Since each of the intervals $I_{x'},J_{x'}$ has diameter strictly less than $2$,
the inequality \eqref{eq:farcenters} forces $tI_{x'}+(1-t)J_{x''}$ to be disjoint from $tI_{x'(y)}+(1-t)J_{x''(y)}$, 
and consequently $t A_{x'}+(1-t) B_{x''}$ is disjoint from $tA_{x'(y)}+(1-t)B_{x''(y)}$.

The set $(tA+(1-t)B)_{y}$ contains both of these disjoint sets, so 
\begin{align*} |(tA+(1-t)B)_y|
%&\ge t|A_y|+(1-t)B_y| + |tA_x+(1-t)B_{x'}| \\ why needed?? March 1
&\ge t|A_{x'(y)}|+ (1-t)|B_{x''(y)}| + t|A_{x'}|+(1-t)|B_{x''}|
\\ &\ge t|A_{x'(y)}|+(1-t)|B_{x''(y)}| + \delta^{1/6}.  \end{align*}
For $x'\in\omega^\ddagger_A$ implies that $x'=x'(z)$ for some $z\in\omega^\ddagger\supset\omega^\dagger$;
the definition of $\omega^\dagger$ includes the condition that $|A_{x'(z)}|\ge\delta^{1/6}$.
Similarly, $|B_{x''}|\ge\delta^{1/6}$.

On the other hand, we know by \eqref{eq:anotherintegralbound} that
\[ \int_{\omega^\ddagger} \big(|(tA+(1-t)B)_y|-t|A_{x'(y)}|-(1-t)|B_{x''(y)}| \big)\,dy=O(\delta^{1/3}). \]
The integrand is a nonnegative function of $y\in\omega^\ddagger$. 
We have shown that for any $y\in\omega^\ddagger$ for which there exists $(x',x'')\in\omega^\times$
with the properties assumed above, this integrand is $\ge\delta^{1/6}$.
Therefore the set of all $y\in\omega^\ddagger$
for which there exists such a pair $(x',x'')$, has measure $O(\delta^{-1/6}\delta^{1/3})=O(\delta^{1/6})$.

Since $|\omega^\ddagger_A|\asymp |\omega^\ddagger_B| \asymp 1$, the set of all ordered pairs $(x',x'')\in\omega^\times$
which give rise to such a $y$ consequently also has measure $O(\delta^{1/6})$.
\end{proof}

This reasoning relied on the fact that $\delta^{1/3}\ll\delta^{1/6}$.
But both of these quantities are dictated by the value of $\eta$ chosen in Lemma~\ref{lemma:gotintervals2}.
The larger quantity, $\delta^{1/6}$ is comparable to $\eta$. The other quantity is determined 
by $\eta$ via the inequality in the localization lemma. It is the presence of the exponent $2$ in a term
on the right-hand side of that inequality which results in the comparatively favorable factor
$\delta^{1/3}\asymp \eta^2$.

Lemma~6.6 of \cite{christyoungest} states the following, with slightly different notation.
Let $n\ge 1$, and denote by $B_R$ a ball of radius $R\in\reals^+$ in $\reals^n$.
\begin{lemma} \label{lemma:threefunctions}
Let $R\in(0,\infty)$. Let $\Lambda\subset(0,1)$ be compact, and let $t\in\Lambda$.
Let $f,g,h: \reals^n\to\complex$ be measurable functions.  Suppose that
\begin{equation}
\big|\set{(x,y)\in B_{R}^2: |f(x)+g(y)+h(tx+(1-t)y)|>\tau}\big|<\rho|B_{R}^2|.
\end{equation}
Then there exists an affine function $L:\reals^n\to\complex$ such that
\begin{equation}
\big|\set{x\in B_R: |f(x)-L(x)| >C\tau}\big|<\eps(\rho)|B_R|,
\end{equation}
where $\eps(\rho)\to 0$ as $\rho\to 0$.
There exist affine functions which satisfy corresponding bounds for $g,h$.
Here $C$ is a constant which depends only on $n,\Lambda$, while $\eps(\rho)$
depends only on $\rho,n,\Lambda$.
\end{lemma}

More accurately, 
Lemma~6.6 of \cite{christyoungest} is stated with $f(x)+g(y)+h(tx+(1-t)y)$ replaced in the
hypothesis by $f(x)+g(y)+h(x+y)$, but the proof given there applies under the modified
hypothesis with no essential changes. Alternatively, the variant stated here can be deduced
by applying the original formulation to functions defined by composing $f,g,h$ with appropriate 
affine transformations.

Invoking this lemma in the present context gives: 
\begin{lemma} \label{lemma:intervalbundles}
There exist a linear function $L:\reals\to\reals$, constants $c,c'\in\reals$,
and a measurable set $\scripte\subset[0,a]$ such that 
\begin{align}
&|\varphi(x)-L(x)-c|=O(1) \ \text{ for all $x\notin\scripte$}
\\
&|\psi(x)-L(x)-c'|=O(1)\ \text{ for all $x\notin\scripte$} 
\\
&|\scripte|=o_\delta(1).
\end{align}
\end{lemma}

\begin{proof}
Apply Lemma~\ref{lemma:threefunctions} with $\tau=2$, $f=t\varphi$, $g=(1-t)\psi$, and $h=t\varphi+(1-t)\psi$.
Then $|f(x)+g(y)+h(tx+(1-t)y)|\le \tau$ unless $(x,y)\in\omega^\times$, 
and we have shown that $|\omega^\times|=O(\delta^{1/12})$.
Therefore the hypotheses of the lemma are satisfied with $\rho=O(\delta^{1/12})$. We conclude that
there exist linear functions $L,L',L''$ and constants $c,c',c''$ such that
\begin{gather*}
\varphi(x)-L(x)-c=O(1), 
\\
\psi(x)-L'(x)-c'=O(1), 
\\
t\varphi(x)+(1-t)\psi(x)-L''(x)-c''=O(1) 
\end{gather*}
for all $x\in[0,a]$ except for a set of measure $o_\delta(1)$. 

Inserting this information into the inequality
\[\big|t\varphi(x)+(1-t)\psi(y) -t\varphi(tx+(1-t)y)-(1-t)\psi(tx+(1-t)y) \big|< 2,\]
which holds for all $(x,y)\in[0,a]^2$ outside an exceptional set of measure $O(\delta^{1/12})$, gives
\[ tL(x)+(1-t)L'(y)-tL''(x)-(1-t)L''(y)+c+c'-c''=O(1)\] 
outside a set of $(x,y)\in[0,a]^2$ having measure $o_\delta(1)$. 
This forces $L=L''=L'$ outside a set whose measure is $o_\delta(1)$.
\end{proof}

This discussion can be extended to $\reals^{1}\times\reals^{d-1}$ for any $d$,
with the intervals $I_x,J_x$ of Proposition~\ref{prop:intervalsKL} replaced by convex subsets 
of $\reals^{d-1}$.

% ??? I need to reword the appropriate lemma/proposition in that reference to give
% the relation between the two associated affine functions

Now consider the sets \[\tilde A=\set{(x,y-L(x)-c): (x,y)\in A} \text{ and } \tilde B = \set{(x,y-L(x)-c'): (x,y)\in B}.\]
Then  \[t\tilde A + (1-t)\tilde B = \set{(x,y-L(x)-tc-(1-t)c'): (x,y)\in tA+(1-t)B}\]
and consequently 
$|\tilde A|=|A|$, $|\tilde B|=|B|$, and $|t\tilde A+(1-t)\tilde B|=|tA+(1-t)B|$.

We have proved:
\begin{lemma} \label{lemma:finallynormalized}
There exists a ball $\scriptb\subset\reals^2$ of finite radius centered at $0$ with the following property.
Let $A,B,t$ satisfy the hypotheses \eqref{hypos} and the normalization \eqref{eq:verticalsupsnormalized}. 
If $\delta$ is sufficiently small then there exist an invertible measure-preserving affine automorphism $\scriptl$
of $\reals^2$ and a vector $v\in\reals^2$ such that 
\begin{equation} |\scriptl(A)\cap \scriptb|=1+o_\delta(1) \ \text{ and } \ 
|\big(\scriptl(B)+v\big)\cap \scriptb|=1+o_\delta(1).  \end{equation}
\end{lemma}
Henceforth we replace $A,B$ by their images $\scriptl(A),\scriptl(B)$. It suffices to prove that these satisfy
the conclusion of Theorem~\ref{thm:BMmult}.

A simple consequence is that if $\scriptb$ is enlarged by a fixed amount, then $A,B$ are contained
entirely in $\scriptb$.
\begin{lemma}
There exists a bounded set $\scriptr\subset\reals^2$ such that whenever
$A,B,t,\scriptl,v$ satisfy the hypotheses and conclusions of Lemma~\ref{lemma:finallynormalized},
the sets $A,B$ are contained in $\scriptr$.
\end{lemma}

\begin{proof}
$\scriptb$ can be taken to be a ball $B(0,R)$ centered at the origin.
Set $R' = t^{-1}(2-t)R$.
If $A$ contains some point $z\notin B(0,R')$, then 
\[ tA+(1-t)B\supset \big(tz+(1-t)(B\cap\scriptb) \big)
\cup \big(t (A\cap\scriptb)+(1-t)(B\cap\scriptb) \big).\]
Since $tR'-(1-t)R=R$, these two sets are disjoint except possibly for a single point of intersection.
Therefore 
\begin{align*} 
|tA+(1-t)B| &\ge (1-t)^d|B\cap\scriptb| +|t (A\cap\scriptb)+(1-t)(B\cap\scriptb)| 
\\ &\ge (1-t)^d(1+o_\delta(1)) + |A\cap\scriptb|^t|B\cap\scriptb)|^{1-t} 
\\ &\ge 1+(1-t)^d+ o_\delta(1),  \end{align*}
which is a contradiction for small $\delta$.
The same reasoning  applies to $B$.
\end{proof}

\section{Step Eight: Rotations}
Let $\scripto$ be an arbitrary element of the group $O(2)$. For any set $E\subset\reals^2$ 
define $E_\scripto=\set{x: \scripto x \in E}$. Then $|E_\scripto|=|E|$.
Since $A_\scripto+B_\scripto = (A+B)_{\scripto}$, if $|t A+(1-t) B|<|A|^t|B|^{1-t}+\delta$
then $|t A_\scripto+(1-t) B_\scripto|<|A_\scripto|^t|B_\scripto|^{1-t}+\delta$, 
so the pair $(A_\scripto,B_\scripto)$ satisfies the hypotheses of our theorem, with the same value of $\delta$
as for $(A,B)$.

Let $A=\scriptl(A)$ and $B=\scriptl(B)$ satisfy the hypotheses and conclusions of Lemma~\ref{lemma:finallynormalized}.
Let $\scripto\in O(2)$ be arbitrary, and consider the pair $(A_\scripto,B_\scripto)$, 
which continues to satisfy the hypotheses \eqref{hypos}.
Apply the above reasoning, from Step 1 through Lemma~\ref{lemma:intervalbundles} but 
without introducing the affine transformation used to obtain Lemma~\ref{lemma:finallynormalized},
New quantities such as $\sup_{x\in\reals}|(A_\scripto)_x|$, an associated parameter $a_\scripto$, a linear
function $L_\scripto$, and parameters $c_\scripto,c'_\scripto$ arise. 

An essential point is that these quantities are now automatically bounded, uniformly in $\scripto$
as well as in $A,B,t$.
\begin{lemma} \label{lemma:rotations} 
Let $(A,B)$ be a pair of compact subsets of $\reals^2$ which satisfies
the hypotheses and 
conclusions of Lemma~\ref{lemma:finallynormalized}. Then uniformly for all $\scripto\in O(2)$,
\begin{align*}
& a_\scripto\asymp 1, 
\\ &\sup_{x\in\reals}|(A_\scripto)_x|\asymp 1  \text{ and } \sup_{x\in\reals}|(B_\scripto)_x|\asymp 1, 
\\ &\sup_{|x|\le 1}|L(x)|= O(1), \text{ and }  c_\scripto,c'_\scripto=O(1). 
\end{align*}
\end{lemma}

\begin{proof}
All of these conclusions follow simply from the facts that
$|A_\scripto\cap\scriptb|=|A\cap\scriptb|\gtrsim 1$  and likewise
$|B_\scripto\cap\scriptb|=|B\cap\scriptb|\gtrsim 1$, where $\scriptb\subset\reals^2$
is the ball mentioned in Lemma~\ref{lemma:finallynormalized}.   

For instance,  since
\[ 1+O(\delta)=|A_\scripto|\le o_\delta(1)+ |A_\scripto\cap\scriptb|
\le o_\delta(1)+C\sup_x|(A_\scripto)_x|, \]
$\sup_{x\in\reals}|(A_\scripto)_x|$ cannot be small.
On the other hand, for any $z\in\reals^1$, 
$tA_\scripto+(1-t)B_\scripto$ contains $\set{tz}\times t(A_\scripto)_z + (1-t)B_\scripto$,
so \[|tA_\scripto+(1-t)B_\scripto|\ge t|(A_\scripto)_z| (1-t)|\pi(B_\scripto)|\] where $\pi:\reals^2\to\reals^1$
is the projection $\pi(x,y)=x$. 
The measure of $\pi(B_\scripto)$ is bounded below since $|B_\scripto\cap\scriptb|$ is bounded below.
\end{proof}

In other words, once $A,B$ have been normalized via appropriately adapted measure-preserving
affine transformations, all rotations of $A,B$ remain normalized.

\section{Step Nine: Precompactness}

For $s\ge 0$ let $H^s$ denote the usual Sobolev space of all functions in $\lt(\reals^2)$ 
which have $s$ derivatives in $\lt$.
\begin{lemma} \label{lemma:FT}
There exists a bounded subset $\scriptr\subset\reals^2$ with the following property.
Let $s\in(0,\tfrac12)$.
For all sufficiently small $\delta>0$,
for any pair of compact subsets $A,B\subset\reals^2$
which satisfy the hypotheses \eqref{hypos},
there exist a measure-preserving affine automorphism $\scriptl$ of $\reals^2$
and a vector $v\in\reals^2$ such that the sets $\tilde A=\scriptl(A)$
and $\tilde B = \scriptl(B)+v$ are contained in $\scriptr$.
Moreover, there exists a decomposition 
\begin{equation} \one_{\tilde A} = g+b \end{equation}
where $g\in H^s(\reals^2)$ and
\begin{align} &\norm{g}_{H^s}=O(1) \text{\  uniformly in $\delta$} 
\\ &\norm{b}_{H^0}=o_\delta(1).  \end{align}
There is a decomposition of $\one_{\tilde B}$ with corresponding properties.
\end{lemma}

\begin{proof}
We have already shown the existence of $\scriptl,v$ which place $\tilde A,\tilde B$
inside a fixed bounded region.
We simplify notation for the remainder of the proof for writing $A,B$ instead of $\tilde A,\tilde B$.

It suffices to show that there exists  a function $\delta\mapsto R(\delta)$
such that $R(\delta)\to\infty$ as $\delta\to 0$ and
\[ \int |\widehat{\one_A}(\xi)|^2 \abr{\min(|\xi|,R(\delta))}^{2s}\,d\xi = O(1) \]
uniformly in $\delta,A$.
To prove this, it suffices to bound
\[ \int |\widehat{\one_A}(\xi)|^2 \abr{\min(|\xi_j|,R(\delta))}^{2s}\,d\xi  \]
for each $j\in\set{1,2}$ where $\xi=(\xi_1,\xi_2)\in\reals^2$.
Since the hypotheses of the lemma are invariant under rotations of $\reals^2$, it suffices to treat $j=2$.

We have shown that there exists a decomposition of $A$ as a disjoint union $A = \scripta\cup \scripte$ 
of Lebesgue measurable sets, where $|\scripte|=o_\delta(1)$ and
for each $x\in[-\lambda,\lambda]$, either $|\scripta_x|=0$ or 
there exists an interval $I_x\subset [-\lambda,\lambda]$ such that 
\begin{gather*} \scripta_x \subset I_x 
\\ |I_x\setminus\scripta_x|=o_\delta(1).  \end{gather*}
These intervals can be taken to depend measurably on $x$, so that the associated set
\[A^*=\set{(x,y)\in\reals^{1}\times\reals^1: |\scripta_x|>0 \text{ and } y\in I_x}\] is Lebesgue measurable.

For $0< u\le 1$ and $f\in\lt(\reals^2)$ define
$\Delta_u f(x_1,x_2) = f(x_1,x_2+u)-f(x_1,x_2)$. 
As a consequence of the fibered structure of $A^*$, the function $f=\one_{A^*}$ plainly satisfies
\begin{equation}
\norm{\Delta_u f}_{\lt} = O(u^{1/2}) 
\end{equation}
uniformly for all $u\in(0,1]$, all $\delta\in (0,1]$, and all sets $A^*$ satisfying the above conditions.  
Since $\widehat{\Delta_u f}(\xi_1,\xi_2) = \big(e^{iu\xi_2}-1\big)\widehat{f}(\xi_1,\xi_2)$,
it follows that
\begin{equation} \int_{\reals^2} \big|e^{iu\xi_2}-1\big|^2\,|\widehat{f}(\xi)|^2\,d\xi = O(u)
\end{equation}
and consequently for any $s\ge 0$,
\begin{equation} \int_{|\xi_2|\asymp u^{-1}} |\xi_2|^{2s} |\widehat{f}(\xi)|^2\,d\xi 
= O(u^{1-2s}).  \end{equation}
It follows that
\begin{equation}
\int_{\reals^2} (1+\xi_2^2)^{s} |\widehat{f}(\xi)|^2\,d\xi = O(1)
\end{equation}
for all $s<\tfrac12$.
Since $\norm{f-\one_A}_2=o_\delta(1)$, this implies that
there exists
a function $\delta\mapsto R(\delta)$ which satisfies $\lim_{\delta\to 0}R(\delta)=\infty$ 
such that
\begin{equation} \label{goodfourierbound} \int_{\reals^2} \langle \min(|\xi_2|,R(\delta))\rangle^{2s}
|\widehat{\one_A}(\xi)|^2\,d\xi = O(1),  \end{equation}
uniformly for all $\delta\in(0,1]$.
\end{proof}

\begin{corollary}
Let $(A_\nu,B_\nu)$ be a sequence of ordered pairs of Lebesgue measurable subsets of $\reals^2$
such that $|A_\nu|\to 1$, $\big|B_\nu|\to 1$, and
$|t A_\nu+(1-t) B_\nu|\to 1$ as $\nu\to\infty$.
Then there exists a sequence of Lebesgue measure-preserving affine transformations $\Phi_\nu$ of $\reals^2$ such that 
the two sequences $(\one_{\Phi_\nu(A_\nu)})$ and $(\one_{\Phi_\nu(B_\nu)})$ 
of indicator functions are precompact in $\lt(\reals^2)$.
\end{corollary}

\begin{proof} We have already shown in Lemma~\ref{lemma:rotations} that 
there exists $\lambda<\infty$ such that
$\Phi_\nu$ can be chosen so that 
$|[-\lambda,\lambda]^2\setminus \Phi_\nu(A_\nu)|\to 0$ 
and likewise 
$|[-\lambda,\lambda]^2\setminus \Phi_\nu(B_\nu)|\to 0$. 
Lemma~\ref{lemma:FT} can then be applied to the
intersections of $\Phi_\nu(A_\nu)$ and $\Phi_\nu(B_\nu)$ with $[-\lambda,\lambda]^2$.
The conclusion then follows from Rellich's Lemma.  \end{proof}

\section{Step Ten: Properties of limits}

Let $(A_\nu,B_\nu) $ be ordered pairs of compact subsets of $\reals^2$
satisfying $|A_\nu|\to 1$, $|B_\nu|\to 1$, and
$|t A_\nu+(1-t) B_\nu|\to 1$ as $\nu\to\infty$. Suppose furthermore
that there exist $F,G\in\lt(\reals^2)$ such that $\one_{A_\nu}\to F$ and $\one_{B_\nu}\to G$
in $\lt(\reals^2)$ norm.  It is elementary that there exist measurable sets $A,B$ 
such that $F=\one_A$ and $G=\one_B$ Lebesgue almost everywhere, and that $|A|=|B|=1$.

Replace each of $A,B$ by the set of all of its Lebesgue points. Every point of each of
these modified sets is a Lebesgue point of that modified set.

\begin{lemma} The pair $(A,B)$ of limiting sets defined above achieves equality 
in the Brunn-Minkowski inequality; that is, 
\begin{equation} |t A+(1-t) B| = 1.  \end{equation} \end{lemma}

\begin{proof}
Set $\tilde A=t A$, $\tilde B=(1-t) B$, and $E=t A+(1-t) B=\tilde A+\tilde B$.
Likewise set $\tilde A_\nu = t A_\nu$ and $\tilde B_\nu = (1-t) B_\nu$.

Define the continuous function $f$ on $\reals^2$ by $f=\one_{\tilde A}*\one_{\tilde B}$.
Each point of $E$ is a Lebesgue point of $E$, and moreover, $f$ is strictly positive at each such point.
Therefore it suffices to prove that for each $\eps>0$,
$E_\eps=\set{x\in E: f(x)>\eps}$ satisfies $|E_\eps|\le 1$.

For any $x\in\reals^2$ and any index $\nu$,
\begin{equation}
\Big|\big(\one_{\tilde A}*\one_{\tilde B}\big)(x)
- \big(\one_{\tilde A_\nu}*\one_{\tilde B_\nu}\big)(x)\Big|
\le C\norm{\one_{A_\nu}-\one_A}_1 + C\norm{\one_{B_\nu}-\one_B}_1,
\end{equation}
which tends to zero as $\nu\to\infty$.
Therefore for any $\eps>0$ 
\[ E_\eps\subset \set{x: \big(\one_{\tilde A_\nu}*\one_{\tilde B_\nu}\big)(x)>\eps}
\subset \tilde A_\nu+\tilde B_\nu\ \text{ for all sufficiently large $\nu$} \]
and therefore
\[ |E_\eps| \le |\tilde A_\nu+\tilde B_\nu|\le 1+o_\nu(1). \]
Taking the limit as $\nu\to\infty$ gives $|E_\eps|\le 1$.
\end{proof}

\begin{corollary} \label{cor:gotoconvex}
Let $(A_\nu,B_\nu)$ be a sequence of ordered pairs of compact subsets of $\reals^2$
satisfying $|A_\nu|\to 1$, $|B_\nu|\to 1$, and $|tA_\nu+(1-t) B_\nu|\to 1$ as $\nu\to\infty$.
Suppose also that the sequences $(\one_{A_\nu})$ and $(\one_{B_\nu})$ are convergent in $\lt(\reals^2)$.
Then there exist a compact convex set $\scriptc\subset\reals^2$ and $c\in\reals^2$ such that
\[ |\scriptc \bigtriangleup A_\nu|+|\scriptc\bigtriangleup (B_\nu-c)|\to 0.  \]
\end{corollary}

\begin{proof}
We have shown that there exist
Lebesgue measurable sets $A,B$ such that $\one_{A_\nu}\to\one_A$
and $\one_{B_\nu}\to\one_B$ in $L^2$ norm, $A$ coincides with the set of all of its Lebesgue
points, likewise for $B$, and $|A|=|B|=1$. We have shown further that 
$(A,B)$ attains equality in the Brunn-Minkowski inequality.
It is known \cite{gardner} that equality is attained if and only if there exists a homothetic pair of compact convex sets
$A^*,B^*$ such that $A\subset A^*$, $B\subset B^*$, and $|A^*\setminus A|=|B^*\setminus B|=0$.
Since $|A|=|B|=1$, $|A^*|=|B^*|=1$. Therefore $B^*$ is a translate of $A^*$.
That $|A^*\bigtriangleup A_\nu|\to 0$ is merely a restatement of the hypothesis
that $\one_{A_\nu}\to\one_A$ in $L^2$ norm, and likewise for $B^*,B_\nu$.
\end{proof}

There exist constants $0<r_0<r_1<\infty$ such that for any convex set $\scriptc\subset\reals^2$
there exists a measure-preserving affine automorphism $L$ of $\reals^2$
such that $\scriptb_0\subset L(\scriptc)\subset \scriptb_1$, where $\scriptb_i$ denotes the ball in $\reals^2$
centered at $0$ with radius $r_i$. 
The hypotheses and conclusion are invariant under affine changes of variables. 
Therefore by replacing each of $A_\nu,B_\nu$ by its image under $L$,
we may assume that the convex set $\scriptc$ in the conclusions of Corollary~\ref{cor:gotoconvex}
satisfies $\scriptb_0\subset \scriptc\subset \scriptb_1$.
Moreover, since $A_\nu,B_\nu$ can be independently translated without altering hypotheses or conclusion,
we may translate $B_\nu$ to ensure that also $|B_\nu\bigtriangleup \scriptc|\to 0$ as $\nu\to\infty$.

\begin{lemma}  \label{lemma:smalldistancetoK}
Let $\scriptr\subset\reals^2$ be a bounded set.
Let $(A_\nu,B_\nu)$ be a sequence of pairs of Borel subsets of $\reals^2$
such that $|A_\nu|\to 1$, $|B_\nu|\to 1$, and $|tA_\nu+(1-t)B_\nu|\to 1$ as $\nu\to\infty$.
Suppose further that $\scriptc_\nu\subset\scriptr$ are convex compact sets
such that $|A_\nu\bigtriangleup \scriptc_\nu|\to 0$
and $|B_\nu\bigtriangleup \scriptc_\nu|\to 0$.
Then
\begin{equation} \distance(A_\nu,\scriptc_\nu)\to 0 \text{ as } \nu\to\infty.  \end{equation}
The same conclusion holds for $B_\nu$. 
\end{lemma}

\begin{proof}
We argue by contradiction.
Let $\rho>1$, and assume that there exist $\eta>0$ and infinitely many indices $\nu$
and points $z_\nu\in A_\nu$ such that $\distance(z_\nu,\scriptc_\nu)\ge \eta>0$.
By passing to a subsequence, we may assume that this happens for every $\nu$.

Since $|\scriptc_\nu|$ is bounded below and its diameter is bounded above, uniformly in $\nu$,
John's theorem guarantees that there exist $r,r'\in\reals^+$ and $c_\nu\in\reals^2$
such that for all sufficiently large $\nu$, $B(c_\nu,r)\subset\scriptc_\nu\subset B(c_\nu,r')$,
and the sequence $c_\nu$ is bounded.

Consider any index $\nu$ and the associated sets $A_\nu,B_\nu,\scriptc_\nu$ and point $z=z_\nu$.
Let $\rho\ge\eta$ be the distance from $z$ to $\scriptc_\nu$.
Consider the closed ball $\scriptb = B(z,\rho)$. Since $\scriptc_\nu$ is convex and $z\notin\scriptc_\nu$,
$\scriptb$ meets $\scriptc_\nu$ at a single point, $w$. The hyperplane tangent to $\scriptb$ at $w$
is a supporting hyperplane for $\scriptc_\nu$; there exists a unit vector $v\in\reals^2$ 
perpendicular to this hyperplane such that $\langle v,x\rangle\le \langle v,w\rangle$ for every $x\in\scriptc_\nu$, 
and $\langle v,z\rangle=\langle v,w\rangle + \rho$.

By the Brunn-Minkowski inequality, \[|t(A_\nu\cap\scriptc_\nu)+(1-t)(B_\nu\cap\scriptc_\nu)|\ge 1-o_\nu(1)\]
since $|A_\nu\cap\scriptc_\nu|,|B_\nu\cap\scriptc_\nu|=1+o_\nu(1)$. Since $\scriptc$ is convex,
this bound is equivalent to
\[|\big(tA_\nu+(1-t)B_\nu\big)\cap \scriptc_\nu|\ge 1-o_\nu(1).\]

Consider the sets 
\begin{align*}
\tilde \scriptc_\nu &= \set{x\in \scriptc_\nu:  \langle v,x\rangle \ge \langle v,w\rangle -t(1-t)^{-1}\rho }  
\\ \tilde B&= \tilde \scriptc_\nu\cap B_\nu.  \end{align*}
If $b\in \tilde B$ then
\begin{align*}
\langle v,\,tz+(1-t)b\rangle 
&= t(\langle v,w\rangle+\rho) +(1-t)\langle v,b\rangle
\\
&> t\langle v,w\rangle+t\rho +(1-t)\big(\langle v,w\rangle-t(1-t)^{-1}\rho\big)
\\ & \ge  \langle v,w\rangle. 
\end{align*}
Therefore $tz+(1-t)\tilde B$ is disjoint from $(tA+(1-t)B_\nu)\cap\scriptc_\nu$, so 
\begin{align*}
|tA_\nu+(1-t)B_\nu| &\ge 1-o_\nu(1)+|tz+(1-t)\tilde B| 
\\
&\ge 1-o_\nu(1) + |tz+(1-t)\tilde\scriptc_\nu| - (1-t)^2|\scriptc_\nu\setminus B_\nu|
\\
&\ge 1-o_\nu(1) + (1-t)^2|\tilde\scriptc_\nu| 
\end{align*}

Because $\scriptc_\nu$ contains both $w$ and $B(0,r)$,
it contains their convex hull. It is elementary that the Lebesgue measure of the set of points $x$ in this
convex hull which satisfy $\langle v,x\rangle\ge -\tau$ is $\ge c\tau^2$ for any sufficiently small $\tau>0$.
Therefore $(1-t)^2|\tilde\scriptc_\nu|$ is bounded below by a certain positive constant, which depends on $\rho,t$
but is independent of $\nu$. 
Therefore $|tA_\nu+(1-t)B_\nu|$ cannot tend to $1$ as $\nu\to\infty$, which is a contradiction.
\end{proof}

\begin{proof}[Conclusion of proof of Theorem~\ref{thm:BMmult}]
We have proved that for any compact set $\Lambda\subset(0,1)$
there exists a function $\delta\mapsto \rho(\delta)$
satisfying $\lim_{\delta\to 0}\rho(\delta)=0$,
with the following property. For any $t\in\Lambda$, for any compact sets $A,B\subset\reals^2$ 
which satisfy  $|tA+(1-t)B|<|A|^t|B|^{1-t}+\delta$ and $|A|,|B|=1+O(\delta)$,
there exist a compact convex set $\scriptc$ satisfying $|\scriptc|=1$,
points $u,v\in\reals^2$, and a ball $\scriptb\subset\reals^2$ centered at $0$
of fixed positive radius for which 
$A\subset \scriptc+\rho(\delta)\scriptb+u$,
and $B\subset \scriptc+\rho(\delta)\scriptb+v$.
Now $\scriptk=\scriptc+\rho(\delta)\scriptb$ is itself a convex set, which contains $\scriptc$ and satisfies
\[|\scriptk| \le (1+O(\rho(\delta))|\scriptc|\le 1+o_\delta(1).\]
Thus all conclusions of Theorem~\ref{thm:BMmult} have been established, under the supplementary assumption
that $A,B$ are compact.

Thus far, it has been assumed that $A,B$ are compact.
To treat general Borel sets $A,B$, choose compact subsets $\tilde A,\tilde B$
satisfying $|A\setminus\tilde A|<\delta$ and $|B\setminus \tilde B|<\delta$.
Apply the result proved above to the pair $(\tilde A,\tilde B)$ to obtain a convex set $\scriptc$
and ball $\scriptb$ as above. In particular, $\scriptc$ contains translates of $\tilde A,\tilde B$,
but not necessarily translates of $A,B$. However,
since Lemma~\ref{lemma:smalldistancetoK} applies to Borel sets, it follows as before that 
there exists a convex set $\scriptk$ containing $\scriptc$ for which
$A\subset \scriptk+u$, $B\subset \scriptk+v$, and $|\scriptk| \le  1+o_\delta(1)$.
\end{proof}


\begin{thebibliography}{20}
%\bibitem{beckner} 
%W.~Beckner, {\em Inequalities in Fourier analysis}, Ann. of Math. (2) 102 (1975), no. 1, 159--182 

%\bibitem{BCCT1} J.~Bennett, A.~Carbery, M.~Christ, and T.~Tao,
%{\em The Brascamp-Lieb inequalities: finiteness, structure and extremals}, 
%Geom. Funct. Anal. 17 (2008), no. 5, 1343--1415

%\bibitem{brascamplieb} H.~J.~Brascamp and E.~H.~Lieb, 
%{\em Best constants in Young's inequality, its converse, and its generalization to more than three functions}, 
%Advances in Math. 20 (1976), no. 2, 151--173

%\bibitem{BLL} H.~J.~Brascamp, E.~Lieb and J.~M.~Luttinger, 
%{\em A general rearrangement inequality for multiple integrals},
%J. Functional Analysis 17 (1974), 227--237

\bibitem{burchard} A.~Burchard,
{\em Cases of equality in the Riesz rearrangement inequality}, 
Ann. of Math. (2) 143 (1996), no. 3, 499--527

\bibitem{youngdiscrete}
M.~Charalambides and M.~Christ,
{\em Near--extremizers for Young's inequality for discrete groups},
preprint, math.CA arXiv:1112.3716

%\bibitem{christparaboloidconvolutionextremizers} M.~Christ,
%{\em On extremals for a Radon-like transform}, preprint, math.CA arXiv:1106.0728 

\bibitem{christradon} M.~Christ, 
{\em Extremizers of a Radon transform inequality}, preprint  math.CA arXiv:1106.0719,
to appear, proceedings of 2011 conference in honor of E.~M.~Stein 

\bibitem{christrieszsobolev} \bysame
{\em An approximate inverse Riesz-Sobolev rearrangement inequality}, preprint,
math.CA arXiv:1112.3715

\bibitem{christyoungest} \bysame
{\em On near-extremizers for Young's inequality for $\reals^d$}, preprint,
math.CA arXiv:1112.4875

\bibitem{christbmhigh} \bysame
{\em Near equality in the Brunn-Minkowski inequality}, preprint

\bibitem{freiman1959}
G.~A.~Fre{\u\i}man,
{\em The addition of finite sets. I}, (Russian) 
Izv. Vys\v S. U\v cebn. Zaved. Matematika 1959 no. 6 (13), 202--213

\bibitem{freiman1962}
\bysame
%G.~A.~Fre{\u\i}man,
{\em Inverse problems of additive number theory. VI. On the addition of finite sets. III.} 
(Russian) Izv. Vys\v S. U\v cebn. Zaved. Matematika 1962 no. 3 (28), 151--157

\bibitem{freiman}
\bysame
%G.~A.~Fre{\u\i}man,
{\em Structure theory of set addition}, Ast\'erisque No. 258 (1999), xi, 1--33

\bibitem{FMP2010}
Figalli, A.; Maggi, F.; Pratelli, A. 
{\em A mass transportation approach to quantitative isoperimetric inequalities}, 
Invent. Math. 182 (2010), no. 1, 167--211

\bibitem{FMP2008}
Fusco, N.; Maggi, F.; Pratelli, A. 
{\em The sharp quantitative isoperimetric inequality}, 
Ann. of Math. (2) 168 (2008), no. 3, 941--980

\bibitem{gardner}
R.~J.~Gardner, {\em The Brunn-Minkowski inequality}, 
Bull. Amer. Math. Soc. (N.S.) 39 (2002), no. 3, 355--405.

%\bibitem{levsm} V.~Lev and P.~Y.~Smeliansky, 
%{\em On addition of two distinct sets of integers}, Acta Arith. 70 (1995), no. 1, 85--91

\bibitem{HO}
H.~Hadwiger and D.~Ohmann, 
{\em Brunn-Minkowskischer Satz und Isoperimetrie},  
Math. Z. 66 (1956), 1--8

\bibitem{henstock}
R.~Henstock and A.~M.~ Macbeath, 
{\em On the measure of sum-sets. I. The theorems of Brunn, Minkowski, and Lusternik},
Proc. London Math. Soc. (3) 3, (1953), 182--194

%\bibitem{lieb} E.~Lieb, {\em Sharp constants in the Hardy-Littlewood-Sobolev and related inequalities}, 
%Ann. of Math. (2) 118 (1983), no. 2, 349--374

%\bibitem{liebloss} E.~H.~Lieb and M.~Loss, 
%{\em Analysis}, Amer. Math. Soc., Providence, RI, 1997

\bibitem{riesz}
F.~Riesz,
{\em Sur une in\'egalit\'e int\'egrale}, Journal London Math. Soc. 5 (1930), 162--168

\bibitem{sobolev}
S.~L.~Sobolev,
{\em On a theorem of functional analysis}, Mat. Sb. (N.S.) 4 (1938), 471--479,
A. M. S. Transl.  Ser. 2, 34 (1963), 39-68.

% The Riesz and Sobolev references are copied from Lieb's HLS paper; I can't find either one in Math Reviews


\bibitem{taovu}
T.~Tao and V.~Vu,
{\em Additive Combinatorics},
Cambridge Studies in Advanced Mathematics, 105. Cambridge University Press, Cambridge, 2006

\end{thebibliography}
\end{document}